\documentclass[12pt]{amsart}
\usepackage{graphicx} 
\usepackage{amsmath,amsthm}
\usepackage{graphicx}
\usepackage[all,cmtip]{xy}
\usepackage{comment}
\usepackage{amsmath}
\usepackage[dvipsnames]{xcolor}

\usepackage{amssymb}
\usepackage{graphicx} 
\usepackage{amsmath}
\usepackage{amssymb}
\usepackage{epsfig}
\usepackage{graphicx}
\usepackage{color}
\usepackage{fullpage}
\definecolor{shadecolor}{gray}{0.875}
\usepackage{amscd}
\usepackage{mathtools}
\usepackage{hyperref}
\usepackage{comment}

\def\PP{{\mathbb P}}

\newcommand{\cO}{\mathcal{O}}

\newtheorem{lemma}{Lemma}[section]
\newtheorem{theorem}[lemma]{Theorem}

\newtheorem{corollary}[lemma]{Corollary}
\newtheorem{proposition}[lemma]{Proposition}

\newtheorem{question}[lemma]{Question}

\author[R. Beheshti]{Roya Beheshti}
\address{Department of Mathematics\\Washington University in St. Louis, MO 63105}
\email{beheshti@wustl.edu}

\author{SHIBASHIS MUKHOPADHYAY}
\address{Department of Mathematics\\Washington University in St. Louis, MO 63105}
\email{m.shibashis@wustl.edu}

\begin{document}
\title{Higher-order tangent lines to smooth hypersurfaces}
\begin{abstract}
    We prove that for every smooth hypersurface $X\subset\PP^n$ of degree $d\geq n\geq5$ over an algebraically closed field of characteristic zero, the space of pointed lines $(\ell,p)$ having intersection multiplicity at least six with $X$ at $p$ has the expected dimension $2n-7$. The proof uses deformations of ruled surfaces preserving prescribed contact with $X$. As an application, we show that a smooth hypersurface of degree $d\geq6$ in $\PP^5$ contains at most $225d^3-1370d^2+1800d$ planes. This bound has optimal order of growth in $d$.
\end{abstract}

\maketitle

\section{Introduction}
Let $X\subset \PP^n$ be a smooth hypersurface of degree $d$ over an
algebraically closed field of characteristic $0$. For $1 \leq k \leq d+1$, let $V_k(X)$ denote the
subscheme of the space of pointed lines in $\PP^n$  parametrizing pairs $(\ell,p)$
such that $\ell$ has intersection multiplicity at least $k$ with $X$
at $p$. The expected dimension of $V_k(X)$  is $2n-k-1$, and every irreducible
component has dimension at least this large. We ask
whether smoothness forces equality in the following range.

\begin{question}
\label{question-main}
Fix $k\geq 1$. If $X$ is smooth and $d\geq n\geq k-1$, does $V_k(X)$
have the expected dimension?
\end{question}

When $k=d+1$, the contact condition forces $\ell \subset  X$, so $V_{d+1}(X)$ is the universal family over the Fano variety of lines on $X$. The Debarre--de Jong conjecture predicts that for every
smooth hypersurface $X\subset\mathbb P^n$ of degree
$d\leq n$ the Fano variety of lines $F(X)$ has the expected
dimension $2n-d-3$. Therefore, for each fixed degree $d$, an affirmative answer to the above question when $n=d$ and $k=d+1$ would imply the Debarre--de Jong conjecture for smooth hypersurfaces of degree $d$ in $\PP^d$, and hence for all smooth Fano hypersurfaces of degree $d$ \cite{beheshti-crelle}. Lines having prescribed high-order contact and the associated
contact cones were studied by Jennings \cite{Jennings1986}.
For general hypersurfaces, Bastianelli, Ciliberto, Flamini,
and Supino \cite{BCFS2023} obtained dimension results for
these cones and applications to measures of irrationality.

The answer is affirmative when $k\leq 4$. Patel, Riedl, Smith, and Tseng \cite{PRST} proved 
that $V_5(X)$ has the expected dimension whenever $d\geq n\geq 4$.
In this paper, we prove the next case.

\begin{theorem}\label{main}
Let $X\subset\PP^n$ be a smooth hypersurface of degree $d$, with
$d\geq n\geq 5$. Then $V_6(X)$ has the expected dimension.
\end{theorem}

Combining the theorem with an intersection-theoretic
calculation from \cite{PRST} gives an explicit bound on the number of
$2$-planes contained in a smooth hypersurface in $\PP^5$:

\begin{corollary}
If $X\subset\PP^5$ is a smooth hypersurface of degree $d\geq 6$,
then the number of $2$-planes contained in $X$ is at most
\[
225d^3-1370d^2+1800d.
\]
\end{corollary}

For $d \geq 3$, the number of $2$-planes in a smooth hypersurface of degree $d$ in $\PP^5$ is finite \cite{beheshti-IMRN}. The corollary improves the
$O(d^4)$ bound of \cite{PRST} to a bound with optimal order of
growth: the Fermat hypersurface of degree $d\geq 3$ in $\PP^5$
contains at least $15d^3$ planes. The use of high-contact loci to bound the number of linear
spaces in smooth hypersurfaces was also discussed by
Starr \cite{Starr2006}. For
comparison, Degtyarev, Itenberg, and Ottem \cite{DIJ} proved, using
lattice theory, that a smooth complex cubic fourfold contains at most
$405$ planes.

The idea of the proof is inspired by \cite{beheshti-crelle}, where it is proved that the Fano variety of lines on every smooth Fano hypersurface of degree $\leq 6$ has the expected dimension. The case $d=n=5$ of our
theorem follows from this result. As we show in Lemma~\ref{d=6}, the case
$d=6$ also follows directly from the known case $d=n=6$ of the Debarre-de Jong conjecture. For $d\geq 7$, the same construction no longer
reduces the problem to the Fano scheme of lines on an auxiliary
hypersurface. Instead, we adapt the deformation-theoretic approach
of \cite{beheshti-crelle} to the cones swept out by lines having intersection
multiplicity at least $6$ with $X$.

Section 2 recalls preliminary results on lines with high-order contact. In Section 3, we reduce the theorem to excluding a four-dimensional family of lines in $\mathbb \PP^5$ having contact order at least six with $X$ and sweeping out a hypersurface $Z$. In Section 4, we study deformations of the cones swept out by lines of the family through a general point of $Z$, and complete the proof by a cohomological vanishing calculation on the associated ruled surfaces.

\subsection*{Acknowledgment}
We are grateful to Raymond Cheng and Eric Riedl for helpful conversations. Roya Beheshti was partially supported by Simons Foundation grant 00007742.

\subsection*{AI disclosure:} The core ideas and all proofs in this work were developed by the authors, inspired by \cite{beheshti-crelle}. ChatGPT was used during the preparation of the paper solely for proofreading and expository assistance. During this review, a gap in Section 3 was identified, which was subsequently resolved by the authors.

\section{Preliminaries}

In this section, we recall several facts about lines with high-order
contact. Throughout, we work over an algebraically
closed field of characteristic $0$.

We write $F(X)$ for the Fano variety of lines on $X$. Let $X\subset\PP^n$ be a smooth hypersurface of degree $d\geq2$,
and let $p\in X$. Choose homogeneous coordinates so that
$p=(1:0:\cdots:0)$, and write a defining equation of $X$ as
\[
f=x_0^{d-1}f_1+x_0^{d-2}f_2+\cdots+f_d,
\]
where each $f_i$ is homogeneous of degree $i$ in $x_1,\ldots,x_n$.
For $1\leq k\leq d+1$, the scheme of lines through $p$ having
contact of order at least $k$ with $X$ is identified with
\[
\Sigma_k^p(X)=V(f_1,\ldots,f_{k-1})\subset\PP^{n-1}.
\]
The lines contained in $X$ are
understood to have infinite contact order. In particular, the scheme
of lines contained in $X$ and passing through $p$ is
\[
F^p(X)=V(f_1,\ldots,f_d).
\]

A basic question is whether these contact equations meet properly
when $p$ is general. The first nonlinear condition occurs for
contact of order at least $3$. Since $X$ is smooth, $f_1\neq0$,
and the restriction of $f_2$ to the hyperplane $V(f_1)$ represents
the second fundamental form $\mathrm{II}_p$ of $X$ at $p$.
This form is nondegenerate at a general point. More generally,
we have the following degeneracy bound, recalled in
\cite[Theorem 3.1]{PRST}.

\begin{proposition}\label{secondfundamentalform}
For $0\leq i\leq n-1$, let
$
R_i=\{p\in X:\operatorname{rank}\mathrm{II}_p\leq n-1-i\}.
$
Then
\[
\operatorname{codim}_X R_i\geq i.
\]
\end{proposition}

We next derive a consequence of Landsberg’s theorem \cite{Landsberg}, whose formulation in \cite[Theorem 1.1]{PRST} we use below.

\begin{proposition}\label{degree-lines-through-point}
Let $X$ be a reduced, irreducible hypersurface in $\PP^n$. Suppose that $G$ is an irreducible component of $F(X)$ whose lines sweep out
$X$. Let $p$ be a general point of $X$ and let $\Sigma$ be an
irreducible component of
$
G^p=\{\ell\in G:p\in\ell\}.
$
Set $r = \dim \Sigma$ and $m=n-1-r$. Then
\[
\deg\Sigma\leq m!.
\]
If equality holds, then $\Sigma$ is a complete intersection of
type $(1,2,\ldots,m)$ in $\PP^{n-1}$.
\end{proposition}

\begin{proof}
Let $d=\deg X$. Choose homogeneous coordinates such that
$p=(1:0:\cdots:0)$, and write a defining equation of $X$ as
\[
f=x_0^{d-1}f_1+x_0^{d-2}f_2+\cdots+f_d,
\]
where each $f_i$ is homogeneous of degree $i$ in
$x_1,\ldots,x_n$.
Since $G$ is an irreducible component of $F(X)$ and $p$ is a general point of $X$, $\Sigma$ is an irreducible component of
$F^p(X)=V(f_1,\ldots,f_d)$. We claim that $\Sigma$ is an irreducible component of
$V(f_1,\ldots,f_m)$. This is immediate if $m=d$. If $m<d$ and
the claim fails, there is an irreducible component $\Sigma'$ of
$V(f_1,\ldots,f_m)$ that properly contains $\Sigma$, so
\[
\dim\Sigma'\geq r +1.
\]
Every line parametrized by $\Sigma'$ has contact of order at least
$m+1$ with $X$ at $p$. The expected dimension of this space is 
$$ n-(m+1)=r.$$
Since $p$ is a general point of $X$ and since $\Sigma'$ is larger than expected, by \cite[Theorem 1.1]{PRST}, all lines parametrized by $\Sigma'$
are contained in $X$. Thus $\Sigma'\subset F^p(X)$,
contradicting the fact that $\Sigma$ is an irreducible component
of $F^p(X)$.

The degree bound now follows since any irreducible component of $V(f_1, \dots, f_m)$ has degree $\leq m!$, and if the degree is equal to $m!$, this component should be equal to the complete intersection $V(f_1, \dots, f_m)$. 

\end{proof}

Let $V_k(X)$ denote the scheme of pointed lines $(\ell,p)$ with
$p\in\ell$ and contact of order at least $k$ with $X$ at $p$.
Patel--Riedl--Smith--Tseng prove that a component of this scheme
has the expected dimension whenever its lines sweep out the
ambient projective space \cite[Theorem 2.8]{PRST}.

\begin{theorem}\label{covering-contact}
Let $W$ be an irreducible component of $V_k(X)$, with
$1\leq k\leq d$. If
\[
\overline{\bigcup_{(\ell,p)\in W}\ell}=\PP^n,
\]
then
\[
\dim W=2n-k-1.
\]
\end{theorem}

The same deformation-theoretic argument gives a nonexistence
statement for high-contact lines covering a subvariety of
$\PP^n$.

\begin{proposition}\label{Ysingularities}
Let $Z\not\subset X$ be an irreducible subvariety of $\PP^n$
of dimension $r$. There is no irreducible family $F$ of lines
sweeping out $Z$ such that a general line parametrized by $F$
is contained in the smooth locus of $Z$ and has contact of
order at least $r+2$ with $X$ at a smooth point of $X\cap Z$.
\end{proposition}

\begin{proof}
We adapt the proof of \cite[Theorem 2.8]{PRST}. Put $k=r+2$.
Since $Z\not\subset X$ and $F$ sweeps out $Z$, a general line
parametrized by $F$ is not contained in $X$. Its contact points
with $X$ are therefore finite. After replacing $F$ by an
irreducible generically finite cover, we may mark a contact point
$p$ of order at least $k$ that is smooth on $X\cap Z$. We restrict
to the open subset on which every line is contained in the smooth locus of 
$Z$.

Consider the family $T$ of all parametrizations
$\iota:\PP^1\xrightarrow{\sim}\ell$ of these marked lines
satisfying $\iota(0)=p$. For a fixed $q\in\PP^1\setminus\{0\}$,
the evaluation map
$
\operatorname{ev}_q:T\longrightarrow Z
$
is dominant, so by generic smoothness, its differential is surjective at a
general point $t\in T$.

Fix such a parametrization and identify $\PP^1$ with $\ell$.
Let $kp\subset\ell$ denote the length-$k$ subscheme supported
at $p$. Since the defining section $f$ of $X$ vanishes on $kp$,
its differential induces a well-defined map
\[
T_Z|_\ell\longrightarrow\cO_\ell(d)|_{kp}.
\]
Since $X\cap Z$ is smooth at $p$, this map is nonzero on the
fiber at $p$ and hence is surjective. 
Its kernel $M$ is a vector
bundle of rank $r$, and
\[
\deg M=\deg(T_Z|_\ell)-k.
\]

We have the short exact sequence

\[
0\longrightarrow\cO_\ell(2)
\longrightarrow T_Z|_\ell
\longrightarrow N_{\ell/Z}
\longrightarrow0.
\]
Since $
N_{\ell/Z}\hookrightarrow
N_{\ell/\PP^n}\simeq\cO_\ell(1)^{\oplus(n-1)}$,  every summand in the splitting of $N_{\ell/Z}$ has degree at
most $1$. Since its rank is $r-1$, we obtain
\[
\deg(T_Z|_\ell)\leq2+(r-1)=r+1,
\qquad
\deg M\leq(r+1)-(r+2)=-1.
\]

Every tangent vector to $T$ at $t$ induces a global section of
$T_Z|_\ell$. Since the family preserves contact of order at
least $k$ at the fixed source point $0$, this section belongs
to $H^0(\ell,M)$. Surjectivity of the differential of
$\operatorname{ev}_q$ therefore gives $r$ such sections whose
values at $q$ are linearly independent. Their exterior product
is a nonzero global section of $\det M$, contradicting
$\deg M<0$.
\end{proof}

We will also use the following classical classification
of projective varieties containing large families of lines.

\begin{proposition}[Segre's classification;
see {\cite[Theorem~1]{Rogora}}]
\label{prop:Segre}
Let $Y\subset\mathbb P^N$ be an integral projective variety
of dimension $m\geq3$ over an algebraically closed field
of characteristic zero. Suppose that
\[
\dim F(Y)\geq2m-3.
\]
Then $Y$ is one of the following:
\begin{enumerate}
\item a linear space;
\item a possibly singular quadric hypersurface in its
linear span;
\item a scroll swept out by a one-parameter family of
$(m-1)$-planes. 
\end{enumerate}
\end{proposition}

\section{Geometric reductions}
In this section, we use the results of the previous section to
rule out some of the simpler cases in the proof of the main
theorem. By the following lemma, it is enough to consider
the case $n=5$.

\begin{lemma}\label{hyperplane-reduction}
Fix $k\geq3$ and $d \geq k-1$. Suppose that Question \ref{question-main} has an affirmative
answer for $n= k-1$.
Then it has an affirmative answer for every $n$ with 
$d\geq n\geq k-1$.
\end{lemma}

\begin{proof}
We prove the statement by induction on $n$. When $n=k-1$, by our hypothesis, $V_k$ has the expected dimension. Suppose that $n>k-1$, and let
$X\subset\PP^n$ be a smooth hypersurface of degree $d\geq n$.
Assume that $W\subset V_k(X)$ is an irreducible component
with $
\dim W>2n-k-1.
$
Let $F$ be its image in the Grassmannian of lines.

Since 
$\dim W>2n-k-1 \geq n-1,$ the lines parametrized by $F$ do not all pass through the same point. 
So by \cite[Lemma 4.4]{beheshti-crelle}, for a general hyperplane $H\subset\PP^n$, the
locus
\[
F_H=\{ [\ell] \in F:\ell\subset H\}
\]
is nonempty and has codimension $2$ in $F$. Hence if 
$W_H=\{(\ell,p)\in W : \ell \subset H\}$, then
$
W_H\subset V_k(X\cap H)$, and 
$$
\dim W_H=\dim W-2>2(n-1)-k-1,
$$
so $\dim V_k(X\cap H)$ is larger than expected. Since $H$ is general, $X\cap H$ is non-singular, contradicting the induction hypothesis.
\end{proof}

Applying the lemma with $k=6$, we may restrict to $n=5$.
If $d=5$, every line having contact of order at least $6$
with $X$ is contained in $X$. Since $\dim F(X)=2$ by \cite{beheshti-crelle},
we have $\dim V_6(X) =3$, as required.
We may therefore assume that $d\geq6$.

Suppose, to the contrary, that $\dim V_6(X)\geq4$.
Let $W$ be an irreducible component of $V_6(X)$ of dimension
greater than expected, let $F$ be its image in the Grassmannian
of lines, and let $Z\subset\PP^5$ be the subvariety swept out
by these lines.

Since $\deg X\geq5$, $X$ is not swept out by lines.
We claim that $\dim F(X)\leq2$. Suppose that $G$ is an
irreducible component of $F(X)$ with $\dim G\geq3$.
Since a surface contains at most a two-dimensional family
of lines, the lines parametrized by $G$ sweep out a divisor
$Y\subset X$. By Proposition~\ref{prop:Segre}, $Y$ is either a linear
space, a quadric, or a scroll swept out by a one-parameter
family of $2$-planes. The first two possibilities are excluded
because every nonzero effective divisor on $X$ has degree
at least $\deg X$ by the Lefschetz theorem. 
The scroll case is also impossible, since a smooth
hypersurface of degree $>2$ in $\PP^5$
contains only finitely many $2$-planes
\cite{Starr2006}.

The locus of pointed lines contained in $X$ therefore has
dimension at most $3$, so a general line parametrized by
$W$ is not contained in $X$. Therefore, 
\[
\dim F=\dim W\geq4,\qquad Z\not\subset X.
\]
By \cite[Theorem 2.8]{PRST}, we also have $Z\neq\PP^5$.
Thus $\dim Z$ is either $3$ or $4$.

We will rule out the following cases:
\begin{enumerate}
\item $\dim Z=3$;
\item $\dim Z=4$ and $F$ is contained in an irreducible
component of $F(Z)$ of dimension at least $5$.
\end{enumerate}

\medskip
\noindent\textbf{Case 1.}
Suppose first that $\dim Z=3$. Since $Z$ contains a
four-dimensional family of lines, it must be a $3$-plane,
and $F$ is its entire family of lines. Set
\[
S=X\cap Z.
\]
The singular locus of $S$ has dimension at most $1$: under
the finite Gauss map of $X$, it maps into the pencil of
hyperplanes containing $Z$. The lines in $Z$ meeting
$\operatorname{Sing}S$ therefore form a family of dimension
at most $3$. Hence a general line of $F$ avoids
$\operatorname{Sing}S$, contradicting
Proposition~\ref{Ysingularities}.

We may therefore assume that $Z$ is a hypersurface in $\PP^5$.

\medskip
\noindent\textbf{Case 2.}  Suppose now that $\dim Z =4$ and $F$ is contained in an irreducible component of $F(Z)$ of dimension $\geq 5$. Since $\dim F(Z) \geq 5$, by Proposition \ref{prop:Segre}, $Z$ is either a linear space, a quadric, or a scroll. We get a contradiction in each case:
\begin{itemize}
\item[(1)] $Z$ is a linear subspace. In this case, $Z\simeq\PP^4$. This possibility is ruled out by Proposition \ref{Ysingularities} since  the intersection $Z\cap X$ has only finitely many singular points, and through every such point there is a $3$-dimensional family of lines in $Z$.

\item[(2)] $Z$ is an irreducible quadric hypersurface. The singular locus of $Z$ is either empty, a point, a line, or a plane. If the singular locus of $Z$ is empty or a point, then the family of lines on $Z$ meeting the singular locus of $Z$ has dimension at most $3$. Moreover, $Z \cap X$ has at most finitely many singular points (see for example \cite[Lemma 2.1]{BR21}).
Therefore, not every line parametrized by $F$ can meet the singular locus of $Z$ or the singular locus of $Z \cap X$. Proposition \ref{Ysingularities} therefore gives a contradiction.

If the singular locus of $Z$ is a line or a plane, then $Z$ contains a $1$-dimensional family of $3$-planes, i.e. $Z$ is a scroll. This is ruled out in the next case.

\item[(3)]  Suppose that $Z$ is a scroll ruled by a one-parameter family
of $3$-planes. Let $G$ be an irreducible component of $F(Z)$
of dimension at least $5$ containing $F$. Then by \cite[Proposition 3.7]{PRST}, there is a ruling
$\{\Lambda_t\simeq\PP^3\}_{t\in T}$ of $Z$ such that $G$
parametrizes the lines contained in its members. In particular,
every line parametrized by $F$ lies in a ruling $3$-plane.

For general $t$, a component $F_t$ of the family of lines of $F$
in $\Lambda_t$ has dimension at least $3$ and sweeps out
$\Lambda_t$. By Proposition~\ref{Ysingularities}, a general
marked contact point of this family lies in
$\operatorname{Sing}(X\cap\Lambda_t)$.
This singular locus has dimension at most $1$, so the only possibility is that $\dim \operatorname{Sing} (X\cap \Lambda_t)=1$, and there is a $1$-dimensional irreducible component $C_t$ of the singular locus such that for every 
point $p\in C_t$, every line
in $\Lambda_t$ through $p$ has contact order at least six with $X$ at $p$.
The next lemma shows this is not possible.
\end{itemize}

\begin{lemma}\label{prop:rigidity-triple-contact}
Let $X=V(f)\subset\mathbb P^n$ be a smooth hypersurface
of degree $d>2$. There are at most finitely many codimension-$2$ linear
subspaces $\Lambda\subset\mathbb P^n$ admitting a curve $C$ in $\Lambda \cap X$ such that $f|_{\Lambda}$ has multiplicity $\geq 3$ at every point of $C$.
\end{lemma}

\begin{proof}

Suppose to the contrary that there is a nonconstant
one-parameter family of such $(n-2)$-planes. After a finite
base change and shrinking the base, we may assume that the
family is parametrized by a smooth curve $T$ and that there
is a flat family of integral projective curves
$\mathcal C\to T$ such that $C_t\subset\Lambda_t$ and
$f|_{\Lambda_t}$ vanishes to order at least three at a
general point of $C_t$.

Fix a general $t\in T$, and write
$\Lambda=\Lambda_t$ and $C=C_t$. Let
\[
\nu:\widetilde C\longrightarrow C
\]
be the normalization, and let
$\varphi:\widetilde C\to\mathbb P^n$ be the composition of $\nu$
with the inclusion of $C$. Set
$L=\varphi^*\mathcal O_{\mathbb P^n}(1)$.

Choose coordinates so that $\Lambda=V(x_0,x_1)$, and set
\[
G_a=\left.\frac{\partial f}{\partial x_a}\right|_\Lambda,
\qquad a=0,1.
\]
Since $X$ is smooth along $C$ and $C$ is in the singular locus of $X \cap \Lambda$, the map
\[
\beta:
\varphi^*\mathcal O_{\mathbb P^n}(1)^{\oplus2}
\longrightarrow \varphi^*\mathcal O_{\mathbb P^n}(d),
\]
given by the pullbacks of $G_0$ and $G_1$, is surjective.
Its kernel $K$ is a line bundle with
$
\deg K=(2-d)\deg L<0,
$
hence $H^0(\widetilde C,K)=0$.

The family $\Lambda_t$ determines a nonzero section
\[
\alpha=(\ell_0,\ell_1)
\in H^0(\Lambda,N_{\Lambda/\mathbb P^n})
\simeq H^0(\Lambda,\mathcal O_\Lambda(1))^{\oplus2}.
\]
Set
\[
\sigma=\ell_0G_0+\ell_1G_1
\in H^0(\Lambda,\mathcal O_\Lambda(d)).
\]
We claim that $\sigma$ vanishes to order at least two
at a general point of $C$. Let $R$ be the local
coordinate ring of $\Lambda$ and $I\subset R$ the ideal
of $C$ near a general point
where $\mathcal C\to T$ is smooth. Set
\[
R_\epsilon=R[\epsilon]/(\epsilon^2),
\]
and let $J\subset R_\epsilon$ be the ideal of the
first-order deformation of $C$, so that
$J\bmod\epsilon=I$. Trivialize
$\mathcal O_\Lambda(d)$ locally, let $g\in R$ represent
$f|_\Lambda$, and let $g_\epsilon\in R_\epsilon$ be the
local equation obtained by restricting $f$ to the
first-order deformation $\Lambda_\epsilon$, so 
\[
g_\epsilon=g+\epsilon\sigma.
\]
The contact hypothesis gives $g_\epsilon\in J^3$.
Since $J\subset IR_\epsilon+(\epsilon)$, we have
\[
J^3\subset
(IR_\epsilon+(\epsilon))^3
=I^3R_\epsilon+\epsilon I^2R_\epsilon.
\]
Comparing coefficients of $\epsilon$ gives $\sigma\in I^2$,
proving the claim.

Therefore, $\sigma$ and all its first partial
derivatives on $\Lambda$ vanish at general points of $C$,
and hence on all of $C$. In particular,
$
\beta\bigl(\nu^*(\alpha|_C)\bigr)
=0.
$
Thus $\nu^*(\alpha|_C)$ is a global section of $K$, so it 
must be zero since $\deg K <0$. It follows that
\[
\ell_0|_C=\ell_1|_C=0.
\]

We now apply the same negative-degree argument to the
coefficients of $\ell_0$ and $\ell_1$. Taking derivative, we get that for every $2\leq i\leq n$,
\[
0=
\left.\frac{\partial\sigma}{\partial x_i}\right|_C
=
\frac{\partial\ell_0}{\partial x_i}\,G_0|_C
+
\frac{\partial\ell_1}{\partial x_i}\,G_1|_C.
\]
Tensoring the sequence defining $K$ by $L^{-1}$ gives
\[
0\longrightarrow K\otimes L^{-1}
\longrightarrow \mathcal O_{\widetilde C}^{\oplus2}
\xrightarrow{\ \beta\otimes L^{-1}\ }
L^{\otimes(d-1)}
\longrightarrow0.
\]
The above identity shows that 
$
(
\frac{\partial\ell_0}{\partial x_i},
\frac{\partial\ell_1}{\partial x_i}
)
$
defines a global section of $K\otimes L^{-1}$.
But
\[
\deg(K\otimes L^{-1})=(1-d)\deg L<0,
\]
so this section is zero. Therefore all coefficients of
the linear forms $\ell_0$ and $\ell_1$ are zero.
This contradicts $\alpha\neq0$.

\end{proof}

Summarizing, we obtain the following proposition.

\begin{proposition}[Reduction to the remaining case]
\label{reduction}
Let $W\subset V_6(X)$ be an irreducible component of dimension
greater than expected, let $F$ be its image in the Grassmannian
of lines, and let $Z\subset\PP^5$ be the subvariety swept out
by these lines. Then $\dim Z=4$, and $F$ is a four-dimensional
irreducible component of $F(Z)$.

Let $p\in Z$ be a general point, and let $\Sigma$ be an
irreducible component of
\[
F^p=\{\ell\in F:p\in\ell\}.
\]
Then $\Sigma$ is a curve in
$\PP(T_pZ)\simeq\PP^3$, and
Proposition~\ref{degree-lines-through-point} gives
\[
\deg\Sigma\leq6.
\]
If equality holds, then $\Sigma$ is a complete intersection
of type $(2,3)$ in $\PP(T_pZ)$.
\end{proposition}

\section{Proof of the main theorem}

In this section we prove the main theorem. We first treat the case $d=6$ using the main result of \cite{beheshti-crelle}. 
\begin{lemma}\label{d=6}
Theorem \ref{main} holds for $d=6$.
\end{lemma}

\begin{proof}
Let $X=V(f)$, and let $\tilde{X} \subset \PP^6$ be the hypersurface defined by $x_6^6=f$. Then $\tilde{X}$ is non-singular. Projection from the point $(0:\dots:0:1)$ gives a morphism $\pi: \tilde{X} \to \PP^5$ branched along $X$. If a  line $\ell$ in $\PP^5$ has intersection multiplicity $6$ at a point with $X$, then $\pi^{-1}(\ell)$ is the union of $6$ lines in $\tilde{X}$. By \cite{beheshti-crelle} $\dim F(\tilde{X}) = 3$, so $\dim V_6(X)=3$.
\end{proof}

From now on we assume $d\geq 7$.  By
Proposition \ref{reduction}, it remains to rule out the
following situation: $X\subset\PP^5$ is a smooth hypersurface
of degree $d\geq7$, and $W\subset V_6(X)$ is an irreducible
$4$-dimensional family whose lines sweep out a hypersurface
$Z\subset\PP^5$. We also assume that the image
of $W$ in the Grassmannian of lines is an irreducible component
$G$ of $F(Z)$.

Since $X$ is not covered by lines, $Z\neq X$. Let $q$ be a general point of $Z$. Then $q \not\in X$ and there is a
one-dimensional family of lines through $q$  having contact order at least
$6$ with $X$. Choose an irreducible component of this family,
and let $S_0$ be the surface swept out by its lines.

The surface $S_0$ is a cone with vertex $q$ over an
irreducible curve $B_0$, obtained by intersecting $S_0$ with
a general hyperplane not containing $q$. Let $\nu:B\to B_0$ be the normalization, and set
\[
\mathcal N=\nu^*\mathcal O_{B_0}(-1).
\]
Normalizing the cone and resolving its vertex gives
the smooth ruled surface
\[
\pi:S=\mathbb P_B(\mathcal O_B\oplus\mathcal N)
\longrightarrow B
\]
and a morphism
\[
\varphi:S\longrightarrow\mathbb P^5.
\]

The morphism $\varphi$ is birational onto $S_0$ and maps each ruling fiber
isomorphically to a line through $q$ having contact order at
least $6$ with $X$ at a marked contact point given by $W$. The contact incidence gives
a (not necessarily smooth) multisection $D$ of $\pi$. 

We next discuss deformations of morphisms from smooth surfaces
to $\PP^n$ that preserve a prescribed order of contact with a hypersurface $X$ along a divisor. We will apply
this discussion to the restriction of $\varphi$ to the smooth surface $S\setminus \operatorname{Sing}(S)$ and the non-singular
divisor $D \setminus \operatorname{Sing}(D)$ obtained above. 
The relation between logarithmic tangent sheaves and
deformations preserving the inverse image of a divisor
is discussed in \cite[\S~2.3]{PatelRiedlTseng2024}.

\subsection{Deformations keeping contact order along a divisor}

Let
        $\varphi:S\longrightarrow \PP^5$
be a generically injective morphism from a smooth surface with $\varphi(S) \not\subset X$. Suppose that $r \geq 1$ and $D\subset S$ is a
a smooth effective integral Cartier divisor such that $\varphi^*X$ has order of vanishing $r$ at a general point of $D$.
Let
$\mathcal L=\varphi^*\cO_{\PP^5}(1).$
Consider the Euler sequence 
$$0 \to \cO_{\PP^5} \to \cO_{\PP^5}(1)^{\oplus 6} \to T_{\PP^5} \to 0.$$ 

Let $X=V(f)$. Since $X$ is smooth, there is a surjective map $\cO_{\PP^5}(1)^{\oplus 6} \to \cO_{\PP^5}(d)$ given by the partial derivatives of $f$. The composition 
$$\cO_{\PP^5} \to \cO_{\PP^5}(1)^{\oplus 6} \to \cO_{\PP^5}(d)$$
is multiplication by $d \; f$, and the composition of the pullback of these maps to $S$ is multiplication by $d \;\varphi^*f$. Since $\varphi^* f|_{rD}=0$, we get a surjection 
$$ \rho: \varphi^*T_{\PP^5} \to \mathcal L^{\otimes d}|_{rD}.$$
Denote the kernel of this map by $M'$. 
Then there is an exact sequence
\begin{equation}
0\longrightarrow M' \longrightarrow \varphi^*T_{\PP^5} 
\xrightarrow{\rho} \mathcal L^{\otimes d}|_{rD}\longrightarrow 0,
\end{equation}
and $M'$ is locally free of rank $5$.

Let $N_{\varphi}$ denote the normal sheaf of $\varphi$, namely the cokernel of the injective map $d\varphi:T_S\longrightarrow \varphi^*T_{\PP^5}$. Let $T'$ and $N'$ be the kernel and image of the map $M'\longrightarrow N_\varphi,$ respectively, so we have a short exact sequence 
$$0 \to T' \to M' \to N' \to 0.$$
Denote by $Q$ the cokernel of the map $$\rho \circ d\varphi: T_S\longrightarrow \mathcal L^{\otimes d}|_{rD}.$$
Then there is also an exact sequence
\[
0\longrightarrow N'\longrightarrow N_\varphi
\longrightarrow Q\longrightarrow0.
\]

\begin{lemma}
\label{lem:Tprime}
We have 
\[
        T'=T_S(-\log D).
\]
\end{lemma}

\begin{proof}

Let $I_D$ be the ideal sheaf of $D$ in $S$.  Since
$\varphi^*X\geq rD$, we have $
\varphi^{*}(f) = s_D^r \alpha$ where $s_D$ is the canonical section of $\cO_S(D)$ and $
\alpha\in H^0(S, \mathcal L^{\otimes d}(-rD)).$
Since we  assume that the intersection multiplicity with $X$ at a general point of $D$ is exactly $r$, we have 
$\alpha|_D\neq 0.$

Consider the composition map 
\[
\delta=\rho\circ d\varphi:
T_S\longrightarrow \mathcal L^{\otimes d}|_{rD}.
\]
Working locally,
let $z=0$ be an equation for $D$ and write
\[
\varphi^*f=z^r\alpha.
\]
For a local section $v$ of $T_S$, we have
\[
\delta(v)
= v(z^r\alpha)\bmod z^r
= r z^{r-1}\alpha\,v(z)\bmod z^r.
\]
So the image
of $\delta$ lies in 
\[
\mathcal L^{\otimes d}\otimes
\frac{I_D^{r-1}}{ I_D^r}
\simeq \mathcal L^{\otimes d}(-(r-1)D)|_{D} \hookrightarrow \mathcal L^{\otimes d}|_{rD}.
\]
So $\delta$ factors as
\[
T_S
\xrightarrow{\nu}
N_{D/S}
\xrightarrow{\;r\alpha|_D\;}
\mathcal L^{\otimes d}(-(r-1)D)|_{D}
\hookrightarrow \mathcal L^{\otimes d}|_{rD},
\]
where
\[
\nu:T_S\longrightarrow N_{D/S}\simeq\cO_D(D)
\]
is the normal map associated with the divisor $D$. 

By definition, the kernel of $\nu$ is the logarithmic tangent
bundle $T_S(-\log D).$
Since $\alpha|_D$ is nonzero, multiplication by
$r\alpha|_D$ is injective.  Therefore
\[
\ker(\delta)=\ker(\nu)=T_S(-\log D),
\]
and the assertion follows.
\end{proof}

Putting everything together, we get the following commutative diagram: 

\begin{equation}\label{deformations}
\xymatrix{
& 0 \ar[d] & 0 \ar[d] & 0 \ar[d] &\\
0 \ar[r] & T'\simeq T_S(-\log D) \ar[r] \ar[d]  & T_S \ar[d] \ar[r] & N_{D/S} \ar[d]^{r\alpha} \ar[r] & 0 \\
0 \ar[r] & M' \ar[r] \ar[d] &\varphi^*T_{\PP^5} \ar[r] \ar[d] &\mathcal L^{\otimes d}|_{rD} \ar[d] \ar[r] & 0 \\
0 \ar[r] & N' \ar[r] \ar[d] &N_{\varphi} \ar[r] \ar[d]  & Q \ar[r] \ar[d] & 0 \\
& 0 & 0 & 0  
}
\end{equation}

\begin{proposition}
\label{prop:deformations}
Suppose that $D$ is smooth and integral and that the general 
multiplicity of $\varphi^*X$ along $D$ is $r$. There is a natural
injective map from the space of 
first-order deformations of $(S,D,\varphi)$ satisfying
$\varphi_\epsilon^*X\geq rD_\epsilon$ to $H^0(S,N')$.
\end{proposition}

\begin{proof}
Let $(S_\epsilon, D_\epsilon,\varphi_\epsilon)$ be such a deformation
over $k[\epsilon]/(\epsilon^2)$. Since $(S,D)$ is a smooth pair,
its deformation is locally trivial (see \cite[\S3.4.4]{Sernesi}). Choose local
trivializations on an open cover $\{U_i\}$ of $S$. On each $U_i$, the first-order deformations of $\varphi$ are given by sections
\[
v_i\in H^0(U_i,\varphi^*T_{\PP^5}).
\]
The contact condition gives $\rho(v_i)=0$, so
$v_i\in H^0(U_i,M')$. On overlaps, the infinitesimal transition maps are
given by 
\[
\theta_{ij}\in
H^0(U_i\cap U_j,T_S(-\log D)),
\]
and
\[
v_i-v_j=d\varphi(\theta_{ij}).
\]
Thus the images of the $v_i$ in $N'$ glue to a global
section. The injectivity is clear since the space of deformations of $(S, \varphi)$ maps injectively to $H^0(S, N_{\varphi})$.

\end{proof}

\subsection{A family of cone surfaces} We now construct the required family of cone surfaces.
Let $G$ be the image of $W$ in the Grassmannian of lines,
and let
$
e:\mathcal I_G\longrightarrow Z
$
be the evaluation morphism from the universal family of
lines over $G$. Since $\dim G=4$, we have
$\dim\mathcal I_G=5$. The morphism $e$ is dominant, so
a general fiber of $e$  has dimension one.

After a finite
extension and possibly shrinking, we get a generically finite dominant map
\[
q:T\longrightarrow Z,
\]
and a smooth projective family of curves
$\mathcal B\to T$ parametrizing lines through $q(t)$.
We may assume that $T$ is smooth and irreducible.

Pulling back the universal line over $G$ gives a smooth
family of ruled surfaces
\[
\mathcal S\longrightarrow T
\]
and an evaluation morphism
\[
\Phi:\mathcal S\longrightarrow Z\subset\PP^5.
\]
For a general $t\in T$, write $S_t$ for the corresponding
surface and $\varphi_t$ for the restriction of $\Phi$. The incidence
construction gives a section $C_t$ of $S_t\to B_t$ contracted
to $q(t)$, and $\varphi_t$ is birational onto the cone swept
out by the corresponding lines.

For a general $t\in T$, $q(t) \not\in X$, so the contact incidence
gives an integral multisection $D_t \subset S_t$ disjoint from $C_t$. Let $r\geq6$ be the
multiplicity of $\varphi_t^*X$ along this multisection for a general $t$.
After a further finite extension and shrinking $T$, we
can assume there is a flat Cartier
divisor
$
\mathcal D\subset\mathcal S
$
such that 
$
\Phi^*X\geq r\mathcal D.
$
We may also assume the multiplicity along
a general point of $D_t$ is exactly $r$.

For a general member
$(S,D,\varphi)$, define $M'$ using the surjection
\[
\rho:\varphi^*T_{\PP^5}\longrightarrow
\mathcal L^{\otimes d}|_{rD},
\]
and let $N'$ be its image in $N_\varphi$, as above. The sheaf 
$T_S(-\log D)$ is locally free of rank $2$ and
\[
\det T_S(-\log D)=\omega_S^{-1}(-D).
\]

\begin{proposition}
\label{prop:normal-nonvanishing}
For a general member of the family constructed above,
\[
H^0 (
S,\bigwedge^4M'\otimes\omega_S(D))\neq0.
\]
\end{proposition}

\begin{proof}
The evaluation morphism $\Phi:\mathcal S\to Z$ is dominant by our assumption. 
Since we work in characteristic zero,
the differential of $\Phi$ has rank $4$ at a general point
of $\mathcal S$. Fix a general point $(t,p)$ in $\mathcal S$, and write $(S,D,\varphi)$ for the corresponding member. The image of
\[
T_{\mathcal S}|_{(t,p)}
\longrightarrow
\varphi^*T_Z|_p
\subset
\varphi^*T_{\PP^5}|_p
\]
is therefore $4$-dimensional.
Since the kernel of
\[
\varphi^*T_{\PP^5}|_p\longrightarrow N_\varphi|_p
\]
is the two-dimensional subspace $T_S|_p$,
the composite
\[
T_{\mathcal S}|_{(t,p)}
\longrightarrow
\varphi^*T_Z|_p
\longrightarrow
N_\varphi|_p
\]
has image of dimension at least $2$.

On the other hand, if we set 
$
U=S\setminus\operatorname{Sing}(D),
$
then the construction
in Proposition~\ref{prop:deformations} gives
a linear map
$$
T_tT\longrightarrow H^0(U,N'|_U).
$$
So we get a commutative diagram 
$$\xymatrix{
T_{\mathcal S}|_{(t,p)} \ar[rr] \ar[d] & &N_{\varphi}|_p \\
T_tT \ar[r] &  H^0(U,N'|_U) \ar[r]   & N'|_p \ar[u]
}
$$
Since the image of the top row is at least $2$-dimensional, the image of the map $H^0(U,N'|_U) \to  N'|_p$ is  at least $2$-dimensional. 
Hence there are sections 
\[
\sigma_1,\sigma_2\in H^0(U,N'|_U)
\]
whose values are linearly independent at a general point. Their exterior product, $\sigma_1 \wedge \sigma_2$, is therefore a generically nonzero
section of $\bigwedge^2(N'|_U)$.
The short exact sequence 
$$ 0 \longrightarrow T_S(-\log D)|_U \longrightarrow M'|_U \longrightarrow N'|_U \to 0$$
induces a map 
\[
\det(T_S(-\log D)|_U)
\otimes\bigwedge^2(N'|_U)
\longrightarrow
\bigwedge^4(M'|_U).
\]
This map is generically injective, and $M'$ is locally free, therefore we get an injective map
\[
\frac{\bigwedge^2(N'|_U)}{\operatorname{tors}}
\hookrightarrow
\bigwedge^4(M'|_U)\otimes
\det (T_S(-\log D)|_U)^\vee
\simeq
(\bigwedge^4M'\otimes\omega_S(D))|_U.
\]
Since $\sigma_1\wedge\sigma_2$ is generically nonzero,
its image gives a nonzero section of
\[
(\bigwedge^4M'\otimes\omega_S(D))|_U.
\]
This bundle is locally free, and $S\setminus U$ is finite.
The section therefore extends to all of $S$,
proving the assertion.

\end{proof}

To get a contradiction and complete the proof of Theorem \ref{main}, it remains to prove the vanishing
\[
H^0(
S,\bigwedge^4M'\otimes\omega_S(D)
)=0
\]
for the family constructed above when $r\geq 6$, which we do in the following section.
\subsection{The vanishing calculation}

Let $F$ be a fiber of $\pi: S \to B$. Set $m=D \cdot F$ and $e=\deg B_0$. Since $\varphi^*X\geq rD$, we have 
\begin{equation}\label{eq:contact-degree}
rm\leq d.
\end{equation}
Since $\varphi$ contracts 
$C$, we have $\mathcal L|_C \simeq \cO_C$. The cone construction gives $N_{C/S} \simeq \mathcal N$ and $C^2 = \deg \mathcal N = -e$. 

Because $\mathcal L^{\otimes m}$ and $\cO_S(D)$ have the same degree on every fiber of $\pi$, $$\cO_S(D) \otimes \mathcal L^{\otimes (-m)} \simeq  \pi^*E$$ for some line bundle $E$ on $B$. Restricting to $C$, since $C \cap D= \emptyset$ and $\mathcal L|_C \simeq\cO_C$, we see that $\pi^*E|_C \simeq \cO_C$, so 
$$\cO_S(D) \simeq \mathcal L^{\otimes m}.$$
Similarly, $\mathcal L \otimes \cO_S(-C)$ has degree zero on every fiber, so it is the pullback of a line bundle on $B$ 
Restricting to $C$ and using the fact that $\mathcal L|_C \simeq \cO_C$ gives
$$\mathcal L \simeq \pi^* \mathcal N^{-1}  \otimes \cO_S(C).$$
We now prove a handful of useful lemmas, before proving the main vanishing theorem in Proposition~\ref{vanishing}
\begin{lemma}\label{push-forward}
We have
\[
\pi_*\cO_{rD}\simeq
\bigoplus_{i=0}^{rm-1}\mathcal N^{\otimes i}.
\]
\end{lemma}

\begin{proof}
Put $N=rm$. Twisting the ideal-sheaf sequence of $rD$ by
$\cO_S((N-1)C)$ gives
\[
0\longrightarrow
\cO_S((N-1)C-rD)
\longrightarrow
\cO_S((N-1)C)
\longrightarrow
\cO_{rD}((N-1)C)
\longrightarrow0.
\]
The left-hand term restricts to
$\cO_{\PP^1}((N-1)-rm)=\cO_{\PP^1}(-1)$ on every ruling
fiber. Therefore \[
\pi_*\cO_S((N-1)C-rD)
=
R^1\pi_*\cO_S((N-1)C-rD)=0.
\]
Since $C\cap D=\emptyset$, we have
\[
\cO_{rD}((N-1)C)\simeq\cO_{rD}.
\]
So pushing forward the above short exact sequence to $B$ we get 
\[
\pi_*\cO_{rD}\simeq\pi_*\cO_S((N-1) C) \simeq
\operatorname{Sym}^{N-1}(\cO_B\oplus\mathcal N)
\simeq
\bigoplus_{i=0}^{N-1} \mathcal N^i.
\]
\end{proof}

Recall that $B_0$ is the base curve of the cone surface $S_0$, and $\nu:B\to B_0$ is its normalization.

\begin{lemma}\label{genus-deg}
If $B$ and $B_0$ are as above, then 
$$ \deg B_0 \geq 2g_{B}-2$$ 
and the equality holds if and only if $B$ is a smooth complete intersection of type $(1,1,2,3)$ in $\PP^5$.
\end{lemma}

\begin{proof}
We have shown before that $B_0$ is an irreducible component  of a complete intersection of a quadric and a cubic in $\PP^3$. 
So $$ \deg B_0 \leq 6.$$ If $\deg B_0 \geq 4$, then $B_0$ is non-degenerate, so by Castelnuovo's bound 
$$ g_B \leq \frac{(\deg B_0- 2)^2}{4}$$ which gives the desired inequality. If $B_0$ has degree $ \leq 3$, then $g_B \leq 1$, so the inequality holds. 

Equality is possible only when $e=6$ and $g_{B}=4$, and in this case, $B_0$ is the entire complete intersection and is non-singular. 
\end{proof}

\begin{proposition}\label{vanishing}
If $d \geq 7$ and $r \geq 6$, then for the family constructed above, $$H^0(
S,\bigwedge^4M'\otimes\omega_S(D)
)=0.$$
\end{proposition}

\begin{proof}
The defining sequence for $M'$ gives
\[
\det M'\simeq \mathcal L^{\otimes 6}(-rD).
\]
Thus Serre duality yields
\begin{align*}
H^0\bigl(S,\wedge^4M'(K_S+D)\bigr)
&\simeq H^0\bigl(S,(M')^\vee\otimes
\mathcal L^{\otimes 6}(K_S-(r-1)D)\bigr)\\
&\simeq H^2\bigl(S,M'\otimes
\mathcal L^{\otimes (-6)}((r-1)D)\bigr)^\vee.
\end{align*}
So it is enough to prove that the last $H^2$ vanishes.

Tensoring the defining sequence for $M'$ by
$\mathcal L^{\otimes(-6)}((r-1)D)$ gives the following short exact sequence 
$$0 \to  M' \otimes \mathcal L^{\otimes(-6)} ((r-1)D) \to \varphi^*T_{\PP^5} \otimes  \mathcal L^{\otimes(-6)} ((r-1)D)  \to \mathcal L^{\otimes(d-6)} ((r-1)D)|_{rD} \to 0.$$
Applying the long exact sequence of cohomology, it is enough to show
\begin{itemize}
\item[(1)] $H^2(S, \varphi^*T_{\PP^5} \otimes  \mathcal L^{\otimes(-6)} ((r-1)D))=0$
\item[(2)]  The map $$H^1(S, \varphi^*T_{\PP^5} \otimes  \mathcal L^{\otimes(-6)} ((r-1)D))  \to H^1(S, \mathcal L^{\otimes(d-6)} ((r-1)D)|_{rD})$$ is surjective.
\end{itemize}
Since $r \geq 6$ and $m\geq 1$, we have $(r-1)m-5 \geq 0$, so the Euler sequence shows the first vanishing holds. So it remains to prove (2). Set $t=d+(r-1)m-6$. We have $$\mathcal L^{\otimes (d-6)}((r-1)D)|_{rD} \simeq \mathcal L^{\otimes t}|_{rD} \simeq \pi^* \mathcal N^{\otimes (-t)}|_{rD},$$ so 
\begin{align*}
    H^1(S, \mathcal L^{\otimes(d-6)} ((r-1)D)|_{rD}) & = H^1(B, \pi_* (\pi^*\mathcal N^{\otimes (-t)} \otimes \cO_{rD})) \\ & = H^1(B, \mathcal N^{\otimes(-t)} \otimes \pi_* \cO_{rD}).
\end{align*}
By Lemma \ref{push-forward}, we have $$\pi_* \cO_{rD} = \bigoplus_{i=0}^{rm-1} \mathcal N^{\otimes i}$$
so $$H^1(S, \mathcal L^{\otimes (d-6)} ((r-1)D)|_{rD}) = H^1(B, \bigoplus_{i=0}^{rm-1} \mathcal N^{\otimes(i-t)}).$$
The largest exponent is $m+5-d$, and $\deg \mathcal N <0$.
Since $rm\leq d$, $r\geq6$, and $d\geq7$, we have
\[
d-m-5\geq1.
\]
If
\[
(d-m-5)e>2g_B-2,
\]
then every summand has vanishing $H^1$,
and the required surjectivity is automatic.

It remains to consider the case $$(d-m-5) e  \leq 2g_B-2.$$
Since $d-m-5 \geq 1$, by Lemma \ref{genus-deg}, if $(d-m-5) e  \leq 2g_B-2$, 
then $B_0$ is a smooth complete intersection of type $(1,1,2,3)$ in $\PP^5$, and $d-m-5=1$. Since $d \geq 7$ by our assumption and $mr \leq d$, this implies $d=7$, $m=1$, and $r \in \{6, 7\}$. We will prove (2) in these two cases. 

Since $m=1$, $\pi|_D:D\to B$ is an isomorphism, so $D$ is a smooth section of $\pi$. 
Since $B$ is a smooth complete intersection of a cubic and a quadric in $\PP^3$, $K_B \simeq \mathcal N^{-1}$, and $\mathcal L \simeq \cO_S(D)$. The required surjectivity is therefore
\[
H^1\bigl(S,\varphi^*T_{\PP^5}\otimes
\mathcal L^{\otimes(r-7)}\bigr)
\longrightarrow
H^1(S,\mathcal L^{\otimes r}|_{rD}).
\]

Since $D$ is smooth, the diagram
\eqref{deformations} is valid on all of $S$.
Tensoring its first two rows by
$\mathcal L^{\otimes(r-7)}$ gives the commutative
diagram with exact rows
$$\xymatrix{ 
0 \ar[r] & T_S(-\log D) \otimes \mathcal L^{\otimes (r-7)} \ar[r] \ar[d] & T_S \otimes \mathcal L^{\otimes (r-7)}  \ar[r] \ar[d] & N_{D/S} \otimes  \mathcal L^{\otimes (r-7)}  \ar[r] \ar[d] & 0  \\
0 \ar[r] & M' \otimes \mathcal L^{\otimes (r-7)}  \ar[r] & \varphi^*T_{\PP^5}\otimes \mathcal L^{\otimes (r-7)}  \ar[r] & \mathcal L^{\otimes r}|_{rD} \ar[r] & 0 
}
$$

We first prove that
\[
H^2(S,T_S(-\log D)\otimes
\mathcal L^{\otimes(r-7)})=0.
\]
Let $F$ be a fiber of $\pi$ and $p=D \cap F$. Since $T_S|_F \simeq  \cO_F \oplus \cO_F(2)$, and since there is a short exact sequence 
\[
0 \longrightarrow T_S(-\log D)|_F \longrightarrow T_S|_F \longrightarrow k(p) \longrightarrow 0
\]
we get 
$$ T_S(-\log D)|_F \simeq \cO_F \oplus \cO_F(1).
$$
Since $r=6$ or $7$, every component of $T_S(-\log D) \otimes \mathcal L^{\otimes (r-7)}|_F$ has degree $>-2$, so by Serre duality the desired $H^2$ vanishing holds. 

The first row of the diagram therefore induces
a surjection
\[
H^1(S,T_S\otimes\mathcal L^{\otimes(r-7)})
\longrightarrow
H^1(D,N_{D/S}\otimes
\mathcal L^{\otimes(r-7)}).
\]
By commutativity, it remains to show that the right
vertical map also induces a surjection on $H^1$.

Write
\[
\varphi^*f=s_D^r\alpha,
\qquad
\alpha\in H^0(S,\mathcal L^{\otimes(7-r)}),
\qquad
\alpha|_D\neq0.
\]
From the proof of Lemma \ref{lem:Tprime}, it follows that the right vertical map factors through 
\[
N_{D/S} \otimes \mathcal L^{\otimes (r-7)} 
\xrightarrow{\;r\alpha|_D\;}
\mathcal L^{\otimes r}(-(r-1)D)|_{D}
\hookrightarrow \mathcal L^{\otimes r}|_{rD}.
\]
Identifying $D$ with $B$ and using
$\mathcal L|_D\simeq K_B$, this factorization becomes
\[
K_B^{\otimes(r-6)}
\xrightarrow{\;r\alpha|_D\;}
K_B
\hookrightarrow
\mathcal L^{\otimes r}|_{rD}.
\]
The first map induces a surjection on $H^1$ since by Serre duality, 
its dual is the map
\[
H^0(B,\mathcal O_B)
\xrightarrow{\;r\alpha|_D\;}
H^0(B,K_B^{\otimes(7-r)}),
\]
which is injective because $\alpha|_D\neq0$.

For the second map, note that $\mathcal{L}^{\otimes r}(-(r-1)D)|_{D}\simeq \mathcal{L}^{\otimes r}\otimes\frac{I_D^{r-1}}{I_D^r}$, the right hand side being the kernel of the natural map $\mathcal{L}^{\otimes r}|_{rD}\to \mathcal{L}^{\otimes r}|_{(r-1)D}$. Using the identification $\mathcal{L}|_D\simeq K_B$, we have the exact sequence
\[
0\longrightarrow K_B
\longrightarrow \mathcal L^{\otimes r}|_{rD}
\longrightarrow \mathcal L^{\otimes r}|_{(r-1)D}
\longrightarrow0.
\]
The pushforward calculation above gives
\[
H^1\!\left(S,\mathcal L^{\otimes r}|_{(r-1)D}\right)
\simeq
\bigoplus_{j=2}^{r}H^1(B,K_B^{\otimes j})
=0,
\]
since $g_B=4$. Thus the second arrow also induces a
surjection on $H^1$. This proves the
required surjectivity and completes the proof.
\end{proof}

\bibliographystyle{siam}
\bibliography{mybib}

\end{document}